\documentclass[a4paper,10pt]{article}
\usepackage{latexsym,amsmath,amsthm,amssymb}
\usepackage{a4wide}
\usepackage{hyperref}
\usepackage{marginnote}
\usepackage{color}
\usepackage{cite}
\hypersetup{
pdftitle={Lp HSW}   
pdfauthor={Van Hoang Nguyen},
colorlinks = true,
linkcolor = magenta,
citecolor = blue,
}

\theoremstyle{plain}
\newtheorem{theorem}{Theorem}[section]

\newtheorem{lemma}[theorem]{Lemma}

\theoremstyle{definition}

\theoremstyle{remark}

\renewcommand{\thefootnote}{\arabic{footnote}}

\def\R{\mathbb R}

\def\al{\alpha}
\def\ge{\geq}
\def\de{\delta}

\def\si{\sigma}
\def\na{\nabla}
\def\lt{\left}
\def\rt{\right}

\def\i0i{\int_0^\infty}

\def\Vol{\text{Vol}}

\def\B{\mathbb B}
\def\H{\mathbb H}

\numberwithin{equation}{section}

\title{The sharp Poincar\'e--Sobolev type inequalities in the hyperbolic spaces $\H^n$}
\author{Van Hoang Nguyen\footnote{Institute of Research and Development, Duy Tan University, Da Nang, Vietnam}
}

\begin{document}
\maketitle



\renewcommand{\thefootnote}{}

\footnote{Email: \href{mail to: Van Hoang Nguyen <vanhoang0610@yahoo.com>}{vanhoang0610@yahoo.com}}

\footnote{2010 \emph{Mathematics Subject Classification\text}: 26D10, 46E35}

\footnote{\emph{Key words and phrases\text}: Poincar\'e--Sobolev inequality, Poincar\'e Gagliardo--Nirenberg inequality, Poincar\'e Morrey--Sobolev inequality, sharp constant, hyperbolic spaces}

\renewcommand{\thefootnote}{\arabic{footnote}}
\setcounter{footnote}{0}

\begin{abstract}
In this note, we establish a $L^p-$version of the Poincar\'e--Sobolev inequalities in the hyperbolic spaces $\mathbb H^n$. The interest of this result is that it relates both the Poincar\'e (or Hardy) inequality and the Sobolev inequality with the sharp constant in $\mathbb H^n$. Our approach is based on the comparison of the $L^p-$norm of gradient of the symmetric decreasing rearrangement of a function in both the hyperbolic space and the Euclidean space, and the sharp Sobolev inequalities in Euclidean spaces. This approach also gives the proof of the Poincar\'e--Gagliardo--Nirenberg and Poincar\'e--Morrey--Sobolev inequalities in the hyperbolic spaces $\mathbb H^n$. Finally, we discuss several other Sobolev inequalities in the hyperbolic spaces $\H^n$ which generalize the inequalities due to Mugelli and Talenti in $\mathbb H^2$.  
\end{abstract}

\section{Introduction}
Given $n\geq 2$, let $\H^n$ denote the hyperbolic space of dimension $n$. We will use the Poincar\'e ball model for the hyperbolic space $\H^n$, i.e., a unit ball $\B^n$ with center at origin of $\R^n$ equipped with the metric $g(x) = \frac{4}{(1-|x|^2)^2} \sum_{i=1}^n d x_i^2$. The corresponding Riemannian volume element is $dV = (\frac{2}{1-|x|^2})^n dx$ and for a measurable set $E\subset \H^n$, we denote by $V(E) = \int_E dV$. Our main result of this note states as follows.

\begin{theorem}\label{Mainresult}
Let $n \geq 4$ and $\frac{2n}{n-1} \leq p < n$. Then for any $u\in W^{1,p}(\H^n)$ it holds
\begin{equation}\label{eq:LpHSM}
\int_{\B^n} |\na_g u|_g^p dV - \lt(\frac{n-1}p\rt)^p\int_{\B^n} |u|^p dV \geq S(n,p)^p \lt(\int_{\B^n} |u|^{\frac{np}{n-p}} dV\rt)^{\frac{n-p}n},
\end{equation}
where $\na_g = (\frac{1-|x|^2}2)^2 \na $ denotes the hyperbolic gradient, $|\na_g u|_ g = \sqrt{g(\na_g u, \na_g u)}$ and $S(n,p)$ is the best constant in the $L^p-$Sobolev inequality in $\R^n$ (see, e.g., \cite{Aubin,Talentia}). Furthermore, equality holds true in \eqref{eq:LpHSM} if and only if $u\equiv 0$.
\end{theorem}

The most interest of the inequality \eqref{eq:LpHSM} is that it connects both the sharp Poincar\'e (or Hardy) inequality and the sharp Sobolev inequality in the hyperbolic space $\H^n$. Let $n \geq 2$ and $p>1$, the sharp Poincar\'e inequality asserts that
\begin{equation}\label{eq:HardyPoincare}
\int_{\B^n} |\na_g u|_g^p dV \geq \lt(\frac{n-1}p\rt)^p \int_{\B^n} |u|^p dV,\quad u\in C^{\infty}_0(\B^n).
\end{equation}
The constant $(\frac{n-1}p)^p$ is sharp and is never attained. This leaves a room for several improvements of the inequality \eqref{eq:HardyPoincare}. Notice that the non achievement of sharp constant does not always imply improvement (e. g., Hardy operator in the Euclidean space $\mathbb R^n, n\geq 2$). However, in the hyperbolic space, the operator $-\Delta_{p,\mathbb H^n} - (\frac{n-1}p)^p = {\rm div}(|\nabla_g \cdot|^{p-2}_g \nabla_g \cdot ) - (\frac{n-1}p)^p$ is subcritical, hence improvement is possible.  For examples, the reader can consult the papers \cite{BDGG,BGG,BG} for the improvements of \eqref{eq:HardyPoincare} by adding the remainder terms concerning to Hardy weights, i.e., the inequalities of the form
\[
\int_{\B^n} |\na_g u|_g^p dV - \lt(\frac{n-1}p\rt)^p \int_{\B^n} |u|^p dV \geq C \int_{\B^n} W |u|^p dV,
\]
for some constant $C >0$ and the weight $W$ satisfying some appropriate conditions. For the case $p=2$, Mancini and Sandeep \cite{MS} proved the following Poincar\'e--Sobolev inequalities in $\H^n$ with $n\geq 3$
\begin{equation}\label{eq:MSHn}
\int_{\B^n} |\na_g u|_g^2 dV - \frac{(n-1)^2}4\int_{\B^n} |u|^2 dV \geq C \lt(\int_{\B^n} |u|^q dV\rt)^{\frac2q},\quad u\in C_0^\infty(\H^n),
\end{equation}
where $2 < q \leq \frac{2n}{n-2}$ and $C$ is constant. The inequality \eqref{eq:MSHn} is equivalent to the Hardy--Sobolev--Maz'ya inequality on the half spaces (see \cite[Section $2.1.6$]{Maz'ya}). Especially, in the case $q =\frac{2n}{n-2}$, we get 
\begin{equation}\label{eq:MSHn2}
\int_{\B^n} |\na_g u|_g^2 dV - \frac{(n-1)^2}4\int_{\B^n} |u|^2 dV \geq C_n \lt(\int_{\B^n} |u|^{\frac{2n}{n-2}} dV\rt)^{\frac{n-2}n},\quad u\in C_0^\infty(\H^n),
\end{equation}
where $C_n$ denotes the sharp constant for which \eqref{eq:MSHn2} holds. It was shown by Tertikas and Tintarev \cite{TT} that if $n\geq 4$ then $C_n$ is attained. Using test function, they show that $C_n < S(n,2)$ where $S(n,2)$ denotes the sharp constant in the $L^2-$Sobolev inequality in $\R^n$. More surprisingly, Benguria, Frank and Loss \cite{BFL} proved that $C_3 = S(3,2)$ and $C_3$ is not attained. The non achievement of $C_3$ was also proved by Sandeep ad Mancini \cite{MS} by a different method. We refer the reader to \cite{LuYang} for the Hardy--Sobolev--Maz'ya inequalities of kind \eqref{eq:MSHn2} for higher order derivatives. Therefore, the inequality \eqref{eq:LpHSM} can be seen as a $L^p$ analogue of the result of Benguria, Frank and Loss on the Hardy--Sobolev--Maz'ya inequality in $\H^3$. 

On the other hand, the inequality \eqref{eq:LpHSM} can be seen as a concrete example in the hyperbolic space of the AB program on the sharp Sobolev inequality in Riemannian manifolds \cite{DH}. Let $(M,g)$ be a complete Riemannian manifold of dimension $n\geq 2$. We denote by $H^{1,p}(M)$ the completion of $C_0^\infty(M)$ under the norm $\|u\|_{H^{1,p}} = (\|\na u\|_{L^p(M)}^p + \|u\|_{L^p(M)}^p)^{1/p}$. We wonder to know that for $\theta \in [1,p]$, is there a constant $B$ such that 
\begin{equation*}
S(n,p)^{\theta}\|u\|_{L^{p^*}(M)}^\theta \leq  \|\na u\|_{L^p(M)}^{\theta} + B \|u\|_{L^p(M)}^\theta \tag{$I_{p,opt}^\theta$}
\end{equation*}
for any $u \in H^{1,p}(M)$? In the case of complete compact Riemannian manifolds, it was proved by Hebey and Vaugon \cite{HVa,HVb}, by Druet \cite{Druet} and by Aubin and Li \cite{AL} that $(I_{p,opt}^{\theta})$ holds for $\theta =\min\{2,p\}$. This solves a long standing conjecture due to Aubin \cite{Aubin}. We refer the reader to the original article by Aubin \cite{Aubin} or to the book by Hebey \cite{Hebey} or the paper by Druet and Hebey \cite{DH} for a complete survey on the compact Riemannian manifolds. In the case of complete non-compact Riemannian manifolds, there is several results in which $(I_{p,opt}^\theta)$ is valid. For example, Aubin, Druet and Hebey \cite{ADH} proved that $(I_{p,opt}^p)$ holds for any $1\leq p < n$ with $B = 0$ on the Cartan--Hadamard manifolds (i.e., complete simply connected Riemannian manifold) satisfying Cartan--Hadamard conjecture. In particular, $(I_{p,opt}^p)$ is valid in the hyperbolic spaces for any $1\leq p < n$. Since the inequality \eqref{eq:LpHSM} relates both the sharp Poincar\'e and sharp Sobolev inequalities, then the constants in \eqref{eq:LpHSM} are sharp and can not be improved. Hence, the \eqref{eq:LpHSM} gives an example in which the sharp second constant $B$ can be explicitly computed. We refer to \cite[Theorem $7.7$]{Hebey} for some other examples in the case $p=2$. Note that, in the hyperbolic space $\H^n$, the following inequality holds
\begin{equation}\label{eq:SobHn}
S(n,2)^2 \lt(\int_{\H^n} |u|^{\frac{2n}{n-2}} dV\rt)^{\frac{n-2}n} \leq \int_{\H^n} |\na_g u|_g^2 dV - \frac{n(n-2)}4 \int_{\H^n} |u|^2 dV.
\end{equation}
The constant $n(n-2)/4$ is sharp when $n\geq 4$. By the result of Benguria, Frank and Loss, this constant is not sharp when $n=3$. In this case, the sharp costant is $(n-1)^2/4 =1$. By this observation, we can not hope the valid of \eqref{eq:LpHSM} for any $p\in [1,n)$. We will see below that \eqref{eq:LpHSM} follows by a pointwise estimate for which the condition $\frac{2n}{n-1} \leq p < n$ is sharp. However, in the case $n=3$, we have $\frac{2n}{n-1} =3 > 2$. Hence, the condition $p\geq \frac{2n}{n-1}$ maybe is not optimal for the valid of \eqref{eq:LpHSM}. So, it is more interesting if we can find the sharp $p_0\in [1,n)$ such that \eqref{eq:LpHSM} holds for $p\in [p_0,n)$. 

Let us explain briefly the method used in the proof of Theorem \ref{Mainresult}. Our proof lies heavily on the symmetric non-increasing rearrangement arguments. More precisely, for any function $u\in W^{1,p}(\H^n)$ we define a function $u^*$ which is non-increasing rearrangement function of $u$ (see the precise definition in Section $2$ below). From this $u^*$ we define two new functions $u^\sharp_g$ on $\H^n$ and $u^\sharp_e$ on $\R^n$ by $u^\sharp_g(x) = u^*(V(B_g(0,\rho(x)))),\quad x\in \B^n$ where $ \rho(x) = \ln \frac{1+|x|}{1-|x|}$ denotes the geodesic distance from $x$ to $0$, and $B_g(0,r)$ denotes the open geodesic ball center at $0$ and radius $r>0$ in $\H^n$, and $
u^\sharp_e(x) = u^*(\si_n |x|^n), \quad x\in \R^n$ where $\si_n$ denotes the volume of unit ball in $\R^n$, respectively. The functions $u^\sharp_g$ and $u^\sharp_e$ has the same decreasing rearrangement function (which is $u^*$), then $\|u^\sharp_g\|_{L^q(\H^n)} = \|u^\sharp_e\|_{L^q(\R^n)} = \|u\|_{L^q(\H^n)}$ for any $q \geq 1$. The key in our proof is a result which compares $\|\na_g u_g^\sharp\|_{L^p(\H^n)}^p$ and $\|\na u_e^\sharp\|_{L^p(\R^n)}^p$. Indeed, we will show that
\[
\|\na_g u_g^\sharp\|_{L^p(\H^n)}^p -\|\na u_e^\sharp\|_{L^p(\R^n)}^p \ge\lt(\frac{n-1}p\rt)^p \|u^\sharp_g\|_{L^p(\H^n)}^p. 
\]
Using the sharp Sobolev inequality in $\R^n$ and the P\'olya--Szeg\"o principle in $\H^n$, we obtain the inequality \eqref{eq:LpHSM}. 

The approach to prove Theorem \ref{Mainresult} above also yields the proofs for the following Poincar\'e--Gagliardo--Nirenberg and Poincar\'e--Morrey--Sobolev inequalities in the hyperbolic space $\H^n$,

\begin{theorem}\label{Maintheorem}
Let $n \geq 4$, $\frac{2n}{n-1} \leq p < n$ and $\al \in (0, \frac{n}{n-p}],$ $\al \not=1$. Then for any $u \in C_0^\infty(\H^n)$, the following inequalities holds.
\item (i) If $\al > 1$, then we have
\begin{equation}\label{eq:GN1}
 \|u\|_{L^{\al p}(\H^n)} \leq GN(n,p,\al)\lt(\|\na_g u\|_{L^p(\H^n)}^p -\lt(\frac{n-1}p\rt)^p \|u\|_{L^p(\H^n)}^p\rt)^{\frac{\theta}p} \|u\|_{L^{\al(p-1)+ 1}(\H^n)}^{1-\theta},
\end{equation}
with $\theta = \frac{n(\al-1)}{\al(np - (\al p+ 1-\al)(n-p))}$.
\item (ii) If $\al \in (0,1)$, then we have
\begin{equation}\label{eq:GN2}
 \|u\|_{L^{\al(p-1)+ 1}(\H^n)} \leq GN(n,p,\al)\lt(\|\na_g u\|_{L^p(\H^n)}^p -\lt(\frac{n-1}p\rt)^p \|u\|_{L^p(\H^n)}^p\rt)^{\frac{\theta}p} \|u\|_{L^{\al p}(\H^n)}^{1-\theta},
\end{equation}
with $\theta =\frac{n(1-\al)}{(\al p + 1 -\al)(n-\al(n-p))}.$ 

The constant $G(n,p,\al)$ which appears in \eqref{eq:GN1} and \eqref{eq:GN2} denotes the sharp constant in the Gagliardo--Nirenberg inequality in $\R^n$ (see, e.g, \cite{DDa,DDb,CNV}).

Suppose that $n \geq 2$ and $p > n$. Then for any function $u \in C_0^\infty(\H^n)$, it holds
\begin{equation}\label{eq:HMS}
\|u\|_\infty^p \leq b_{n,p}^p V(\text{\rm supp }u)^{\frac{p-n}n} \lt(\int_{\B^n} |\na_g u|_g^p dV - \lt(\frac{n-1}p\rt)^p\int_{\B^n} |u|^p dV\rt)
\end{equation}
where $\text{\rm supp } u$ denotes the support of the function $u$, and $b_{n,p}$ is the sharp constant in the Morrey--Sobolev inequality in $\R^n$ (see, e.g., \cite{Talentib}).
\end{theorem}

Similar to \eqref{eq:LpHSM}, the inequalities \eqref{eq:GN1}, \eqref{eq:GN2} and \eqref{eq:HMS} relate both the sharp Poincar\'e inequality and the sharp Gagliardo--Nirenberg and the sharp Morrey--Sobolev inequalities in the hyperbolic spaces $\H^n$, so they can not be improved on the constants. The inequality \eqref{eq:LpHSM} is a special case of \eqref{eq:GN1} with $\al = \frac{n}{n-p}$. The case $p=n$ is not included in Theorems \ref{Mainresult} and \ref{Maintheorem}. In this situation, there are some Hardy--Moser--Trudinger type inequalities (see, e.g., \cite{WY2012,Nguyen2017,Nguyen2017a,MST2013,LY}). We refer the reader to the papers \cite{CMa,CMb,CMc,CMd} for more information about the Gagliardo--Nirenberg inequality in the compact Riemannian manifolds.

The rest of this paper is organized as follows. In Section $2$, we recall some basic facts about the symmetric decreasing rearrangement of function in the hyperbolic space $\H^n$ and prove an important result relating the symmetric decreasing rearrangement of function both in hyperbolic space and Euclidean space (see Theorem \ref{sosanh} below). Section $3$ is devoted to prove Theorems \ref{Mainresult} and \ref{Maintheorem}. In section $4$, we discuss some related Sobolev inequalities in hyperbolic space which generalize the inequalities due to Mugelli and Talenti in $\H^2$ to higher dimension.

\section{Symmetric decreasing rearrangements}
It is now known that the symmetrization argument works well in the setting of the hyperbolic spaces $\H^n$ (see, e.g., \cite{Ba} for a reference on this technique). Let us recall some facts about the rearrangement in the hyperbolic spaces. Let $u: \H^n \to \R$ be a function such that
\[
\Vol_g(\{x\in \H^n\,:\, |u(x)|> t\}) = \int_{\{x\in \H^n\,:\, |u(x)|> t\}} dV < \infty,\quad \forall\, t>0.
\]
For such a function $u$, its distribution function, denoted by $\mu_u$, is defined by
\[
\mu_u(t) = V\{x\in \H^n\, :\, |u(x)| > t\}, \qquad t >0.
\]
The function $(0,\infty)\ni t\mapsto \mu_u(t)$ is non-increasing and right-continuous. Then the decreasing rearrangement function $u^*$ of $u$ is defined by
\[
u^*(t) = \sup\{s >0\, :\, \mu_u(s) > t\}.
\]
Note that the function $(0,\infty) \ni t \to u^*(t)$ is non-increasing. We now define the symmetric decreasing rearrangement function $u_g^\sharp$ of $u$ by
\begin{equation}\label{eq:usharpg}
u^\sharp_g(x) = u^*(V(B_g(0,\rho(x)))),\quad x \in \B^n.
\end{equation}
We also define a function $u^\sharp_e$ on $\R^n$ by
\begin{equation}\label{eq:usharpe}
u^\sharp_e(x) = u^*(\si_n |x|^n),\quad x\in \R^n,
\end{equation}
where $\si_n$ denotes the volume of unit ball in $\R^n$. Since $u$, $u_g^\sharp$ and $u^\sharp_e$ has the same non-increasing rearrangement function (which is $u^*$), then we have
\begin{equation}\label{eq:equalkey}
\int_{\B^n} \Phi(|u|) dV = \int_{\B^n} \Phi(u^\sharp_g) dV = \int_{\R^n} \Phi(u_e^\sharp) dx = \int_0^\infty \Phi(u^*(t)) dt,
\end{equation}
for any increasing function $\Phi: [0,\infty) \to [0,\infty)$ with $\Phi(0) =0$. This equality is a consequence of layer cake representation. Moreover, by P\'olya--Szeg\"o principle, we have
\begin{equation}\label{eq:PSprinciple}
\int_{\B^n} |\na_g u^\sharp_g|_g^p dV \leq \int_{\B^n} |\na_g u|_g^p dV.
\end{equation}

Our next aim is to compare $\|\na_g u_g^\sharp\|_{L^p(\H^n)}^p$ and $\|\na u_e^\sharp\|_{L^p(\R^n)}^p$. For simplifying notation, we denote $v =u^*$. By a straightforward computation, we have
\begin{equation}\label{eq:eulideangra}
\int_{\R^n} |\na u^\sharp_e|^p dx = (n\si_n)^p \int_0^\infty |v'(s)|^p \lt(\frac{s}{\si_n}\rt)^{\frac{(n-1)p}n} ds.
\end{equation}
Note that
\[
V(B(0,\rho(x))) = n\si_n \int_0^{\rho(x)} (\sinh t)^{n-1} dt = \si_n \Phi(\rho(x)), 
\]
where
\begin{equation}\label{eq:Phi}
\Phi(t) = n \int_0^t (\sinh s)^{n-1} ds.
\end{equation}
Note that the function $\Phi: [0,\infty) \to [0,\infty)$ is a diffeomorphism, strictly increasing with $\Phi(0) =0$ and $\lim_{t\to \infty} \Phi(t) = \infty$. The gradient of $V(B_g(0,\rho(x)))$ is then given by
\begin{equation*}
\na_g V(B(x_0,\rho(x))) = n\si_n (\sinh \rho(x))^{n-1}\, \na_g \rho(x).
\end{equation*}
Since $|\na_g \rho(x)|_g =1$ for $x \not=0$, then we get
\begin{align*}
\int_{\B^n} |\na_g u^\sharp_g(x)|_g^p dV &= \int_{\B^n} |v'(V(B_g(0,\rho(x))))|^p\lt(n (\sinh (\rho(x))^{n-1}\rt)^p dV\\
&= n \si_n \int_0^\infty  |v'(V(B_g(0,t)))|^p \lt(n\si_n (\sinh t)^{n-1}\rt)^p (\sinh t)^{n-1} dt\\
&=(n\si_n)^{p} \int_0^\infty |v'(V(B_g(0,t)))|^p \lt(\sinh t\rt)^{p(n-1)} n\si_n (\sinh t)^{n-1}dt.
\end{align*}
Making the change of variable $s = V(B_g(0,t)) = \si_n \Phi(t)$ or $t = \Phi^{-1}(\frac{s}{\si_n})$, we have $ds = n\si_n (\sinh t)^{n-1} dt$ and 
\begin{equation}\label{eq:gradientHn}
\int_{\B^n} |\na_g u^\sharp_g|_g^p dV = (n\si_n)^p \int_0^\infty |v'(s)|^p \lt(\sinh \Phi^{-1}\lt(\frac{s}{\si_n}\rt)\rt)^{p(n-1)} ds.
\end{equation}
Let us define the function $k_{n,p}$ on $[0,\infty)$ by
\[
k_{n,p}(s) = (\sinh \Phi^{-1}(s))^{p(n-1)} - s^{\frac{p(n-1)}n}.
\]
We then obtain from \eqref{eq:eulideangra} and \eqref{eq:gradientHn} that
\begin{equation}\label{eq:equalnorm}
\int_{\B^n} |\na_g u^\sharp_g|_g^p dV = \int_{\R^n} |\na u^\sharp_e|^p dx + (n\si_n)^p \int_0^\infty |v'(s)|^p k_{n,p}\lt(\frac s{\si_n}\rt) ds.
\end{equation}
To proceed, we next find an estimate for $k_{n,p}$ from below. In fact, we have the following results.
\begin{lemma}\label{lowerboundk}
It holds
\begin{equation}\label{eq:lowerboundk}
k_{n,p}(s) \geq \lt(\frac{n-1}n\rt)^p s^p,\qquad s \geq 0,
\end{equation}
for any $p \geq 2$ if $n =2$, and for any $p \geq \frac{2n}{n-1}$ if $n\geq 3$.
\end{lemma}
\begin{proof}
It is enough to prove that
\begin{equation}\label{eq:enough}
F_{n,p}(t) = k_{n,p}(\Phi(t)) - \lt(\frac{n-1}n\rt)^p (\Phi(t))^p \geq 0,\qquad t \geq 0,
\end{equation}
for any $p \geq 2$ if $n =2$, and for any $p \geq \frac{2n}{n-1}$ if $n\geq 3$.

If $n =2$, we have $\Phi(t) = 2 (\cosh t-1)$, and
\[
F_{2,p}(t) = \lt(\frac{\Phi(t)^2}4 + \Phi(t)\rt)^{\frac p2} - \Phi(t)^{\frac p2} -\frac1{2^p} \Phi(t)^p\geq 0,
\]
for any $t \geq 0$ if $p \geq 2$.

Suppose that $n \geq 3$. Differentiating the function $F_{n,p}$ we get
\begin{align*}
F_{n,p}'(t) &= p(n-1) (\sinh t)^{p(n-1)-1} \cosh t - p(n-1) (\sinh t)^{n-1} \Phi(t)^{\frac{p(n-1)}n-1}\\
&\hspace{1cm} -\lt(\frac{n-1}n\rt)^p pn (\sinh t)^{n-1} \Phi(t)^{p-1}\\
&= p(n-1) (\sinh t)^{n-1}\lt((\sinh t)^{p(n-1)-n} \cosh t - \Phi(t)^{\frac{p(n-1)}n-1}-\lt(\frac{n-1}n\rt)^p \Phi(t)^{p-1}\rt)\\
&=: p(n-1) (\sinh t)^{n-1} G_{n,p}(t).
\end{align*}
We continue differentiating the function $G_{n,p}$ to obtain
\begin{align*}
G_{n,p}'(t)&= (p(n-1)-n)(\sinh t)^{p(n-1)-n-1} (\cosh t)^2 + (\sinh t)^{(p-1)(n-1)} \\
&\quad -(p(n-1)-n) (\sinh t)^{n-1}\Phi(t)^{\frac{p(n-1)}n-2} -\lt(\frac{n-1}n\rt)^{p-1} (p-1)n (\sinh t)^{n-1} \Phi(t)^{p-2}.
\end{align*}
Replacing $(\cosh t)^2$ by $1 + (\sinh t)^2$, we simplify the expression of $G_{n,p}'$ as
\begin{align*}
G_{n,p}'(t) &= (p-1)(n-1) (\sinh t)^{(p-1)(n-1)} + (p(n-1)-n)(\sinh t)^{p(n-1)-n-1}\\
&\quad -(p(n-1)-n) (\sinh t)^{n-1}\Phi(t)^{\frac{p(n-1)}n-2} -\lt(\frac{n-1}n\rt)^{p-1} (p-1)n (\sinh t)^{n-1} \Phi(t)^{p-2}\\
&= (p-1)(n-1) (\sinh t)^{n-1} \lt((\sinh t)^{(p-2)(n-1)} -\lt(\frac{n-1}n\rt)^{p-2} \Phi(t)^{p-2} \rt)\\
&\quad + (p(n-1)-n)(\sinh t)^{n-1} \lt((\sinh t)^{p(n-1)-2n} - \Phi(t)^{\frac{p(n-1)}n-2}\rt).
\end{align*}
It is easy to see that
\[
\Phi(t) = n\int_0^t (\sinh s)^{n-1} ds < n\int_0^t (\sinh s)^{n-1}\cosh s\, ds =(\sinh t)^n,\quad t >0,
\]
and
\[
\Phi(t) = n\int_0^t (\sinh s)^{n-1} ds < n\int_0^t (\sinh s)^{n-2}\cosh s\, ds = \frac{n}{n-1}(\sinh t)^{n-1},\quad t >0.
\]
Plugging these previous estimates into the expression of $G_{n,p}'$, we get that $G_{n,p}'(t) > 0$ for any $t >0$. This implies $G_{n,p}(t) > G_{n,p}(0) = 0$ for any $t>0$, or equivalently $F_{n,p}'(t) >0$ for any $t>0$. Consequently, $F_{n,p}(t) >F(0) =0$ for any $t>0$. This completes our proof.
\end{proof}
It is remarkable that the pointwise estimate \eqref{eq:lowerboundk} is sharp in $p$. Indeed, if $n =2$ and $p \in (1,2)$ then a reversed estimate of \eqref{eq:lowerboundk} holds. Suppose that $n \geq 3$ and $p < \frac{2n}{n-1}$ we next show that a reversed estimate of \eqref{eq:lowerboundk} holds for $s$ large enough. Indeed, suppose that $n \geq 4$, we have
\[
\Phi(t) = n2^{1-n} \int_0^t \lt(e^s - e^{-s}\rt)^{n-1} ds = \frac{n2^{1-n}}{n-1} e^{(n-1)t} \lt(1 -\frac{(n-1)^2}{n-3} e^{-2t} + o(e^{-2t})\rt),
\]
as $t \to \infty$. Consequently
\[
\Phi(t)^{\frac{p(n-1)}n} = \lt(\frac{n2^{1-n}}{n-1}\rt)^{\frac{p(n-1)}n} e^{\frac{p(n-1)^2}n t} \lt(1 - \frac{p(n-1)^3}{n(n-3)} e^{-2t} + o(e^{-2t})\rt),
\]
and
\[
\Phi(t)^p =\lt(\frac{n2^{1-n}}{n-1}\rt)^{p} e^{p(n-1) t} \lt(1 - \frac{p(n-1)^2}{n-3} e^{-2t} + o(e^{-2t})\rt)
\]
as $t \to \infty$. Note that
\[
(\sinh t)^{p(n-1)} = 2^{p(1-n)} e^{p(n-1)t} \lt(1 -p(n-1) e^{-2t} + o(e^{-2t})\rt),
\]
as $t \to \infty$. Therefore
\begin{align*}
F_{n,p}(t)&=(\sinh t)^{p(n-1)}- \Phi(t)^{\frac{p(n-1)}n} -\lt(\frac{n-1}n\rt)^p \Phi(t)^p\\
& =2^{p(1-n)} e^{p(n-1)t-2t}\lt(\frac{2p(n-1)}{n-3} - \lt(\frac{n2^{1-n}}{n-1}\rt)^{\frac{p(n-1)}n} e^{2t-\frac{p(n-1)}n t} + o(1)\rt)
\end{align*}
as $t \to \infty$. If $p < \frac{2n}{n-1}$  we then have $\frac{p(n-1)}n < 2$, and $F_{n,p}(t) < 0$ for $t >0$ large enough. Suppose $n=3$, we have
\[
\Phi(t) = \frac38 (e^{2t} -e^{-2t} -4t) = \frac38 e^{2t}(1 -4t e^{-2t} + o(t e^{-2t})),
\]
as $t\to \infty$. Hence
\[
\Phi(t)^p = \frac{3^p}{8^p} e^{2pt} \lt(1 -4pt e^{-2t} + o(e^{-2t}t)\rt),
\]
and
\[
\Phi(t)^{\frac{2p}3} =\lt(\frac38\rt)^{\frac{2p}3} e^{\frac{4p}3t} \lt(1 - \frac{8pt}3 e^{-2t} + o(te^{-2t})\rt),
\]
as $t \to \infty$. Evidently,
\[
(\sinh t)^{2p} = 2^{-2p} e^{2pt} (1 + o(t e^{-2t})),
\]
as $t \to \infty$. Consequently, we get
\begin{align*}
F_{3,p}(t)
&= \frac1{4^p}e^{2(p-1)t}t\lt(4p -4^p\lt(\frac38\rt)^{\frac{2p}3} t^{-1} e^{2t-\frac{2p}3t} +o(1)\rt),
\end{align*}
as $t\to \infty$. Since $p < 3$, then we have $2 - \frac{2p}3 >0$ and hence $F_{3,p}(t) <0$ for $t >0$ large enough.

Combining \eqref{eq:lowerboundk} and \eqref{eq:equalnorm} together, we arrive
\begin{equation}\label{eq:equalnorma}
\int_{\B^n} |\na_g u^\sharp_g|_g^p dV \geq \int_{\R^n} |\na u^\sharp_e|^p dx + (n-1)^p \int_0^\infty |v'(s)|^p s^p ds.
\end{equation}
Making the change of function $w(s) = v(s) s^{\frac1p}$ or equivalently $v(s) = w(s) s^{-\frac1p}$. Differentiating the function $v$, we have
\[
v'(s) = w'(s) s^{-\frac1p} -\frac1p w(s) s^{-\frac1p -1}.
\]
We can readily check that if $a-b \leq 0$, $b\geq 0$ and $p\geq 2$ then 
\[
|a -b|^p \geq |a|^p + |b|^p - p a b^{p-1}
\]
Since $v' \leq 0$, applying the previous inequality we get
\begin{align*}
\int_0^\infty |v'(s)|^p s^p ds &\geq \int_0^\infty |w'(s)|^p s^{p-1} ds + \frac1{p^p} \int_0^\infty w(s)^p s^{-1} ds  -p^{2-p}\int_0^\infty w'(s) w(s)^{p-1} ds\\
&= \int_0^\infty |w'(s)|^p s^{p-1} ds + \frac1{p^p} \int_0^\infty v(s)^p ds,
\end{align*}
here we use integration by parts. Plugging this estimate into \eqref{eq:equalnorma} we get
\begin{align}\label{eq:equalnormb}
\int_{\B^n} |\na_g u^\sharp_g|_g^p dV &\geq \int_{\R^n} |\na u^\sharp_e|^p dx + \frac{(n-1)^p}{p^p} \int_0^\infty |v|^p ds  + (n-1)^p\int_0^\infty |(v(s) s^{\frac1p})'|^p s^{p-1} ds.
\end{align}
Since $v =u^*$ is non-increasing rearrangement function of $u_g^\sharp$, then 
\begin{equation}\label{eq:aa}
\int_0^\infty |v|^p ds = \int_{\B^n} |u_g^\sharp|^p dV.
\end{equation}
Plugging \eqref{eq:aa} into \eqref{eq:equalnormb}, we obtain the main result of this section as follows,

\begin{theorem}\label{sosanh}
Let $p\geq 2$ if $n =2$ and $p\geq \frac{2n}{n-1}$ if $n \geq 3$. It holds
\begin{equation}\label{eq:keysosanh}
\int_{\B^n} |\na_g u^\sharp_g|_g^p dV - \lt(\frac{n-1}p\rt)^p\int_{\B^n} |u_g^\sharp|^p dV \geq \int_{\R^n} |\na u^\sharp_e|^p dx.
\end{equation}
\end{theorem}

Theorem \ref{sosanh} was proved in \cite{Nguyen2017a} in the case $p=n$ as a key to establish several improved Moser--Trudinger type inequalities in the hyperbolic space.


\section{Proof of Theorems \ref{Mainresult} and \ref{Maintheorem}}
In this section, we provide the proof of Theorem \ref{Mainresult} and Theorem \ref{Maintheorem}. Our proof uses Theorem \ref{sosanh} above and the known inequalities in the Euclidean spaces such as the sharp Sobolev, Gagliardo--Nirenberg and Morrey--Sobolev inequalities. Let us recall them here. The sharp Sobolev inequality in the euclidean space was independently proved by Aubin and Talenti \cite{Aubin,Talentia} and has the form
\begin{equation}\label{eq:sharpSobolev}
S(n,p) \|u\|_{L^{p^*}(\R^n)} \leq \|\na u\|_{L^p(\R^n)},\quad u \in C_0^\infty(\R^n),
\end{equation}
for $p \in (1,n)$, $p^* = \frac{np}{n-p}$ and the sharp constant $S(n,p)$ is given by
\[
S(n,p) =\lt[\frac1n \lt(\frac{n(p-1)}{n-p}\rt)^{1 -\frac1p} \lt(\frac{\Gamma(n)}{\Gamma(\frac np)\Gamma(n+1-\frac np) \si_n}\rt)^{\frac1n}\rt]^{-1},
\]
where $\Gamma(x) = \int_0^\infty t^{x-1} e^{-t} dt, x >0$ denotes the usual Gamma function. The family of extremal functions is determined uniquely by the function $u(x) = (1+ |x|^{\frac p{p-1}})^{-\frac{n-p}p}$ up to a translation, diltation and multiplying by constant. 

Let $p \in (1,n)$ and $\al \in (0, \frac{n}{n-p}],$ $\al\not=1$. The sharp Gagliardo--Nirenberg inequalities in $\R^n$ was established by Del Pino and Dolbeault \cite{DDa,DDb} and has the forms:
\begin{description}
\item (i) for $\al > 1$,
\begin{equation}\label{eq:sharpGN}
\|u\|_{L^{\al p}(\R^n)} \leq GN(n,p,\al) \|\na u\|_{L^p(\R^n)}^{\theta} \|u\|_{L^{\al(p-1) +1}(\R^n)}^{1-\theta},\quad u\in C^\infty_0(\R^n),
\end{equation}
with $\theta =\frac{n(\al-1)}{\al(np - (\al p+ 1-\al)(n-p)}$, the sharp constant $GN(n,p,\al)$ is given by
\[
GN(n,p,\al) =\lt(\frac{q-p}{p\sqrt{\pi}}\rt)^{\theta} \lt(\frac{pq}{n(q-p)}\rt)^{\frac\theta p} \lt(\frac{\de}{pq}\rt)^{\frac1{\al p}} \lt(\frac{\Gamma(q\frac{p-1}{q-p}) \Gamma(\frac{n}2+1)}{\Gamma(\frac{p-1}{p} \frac{\de}{q-p})\Gamma(n\frac{p-1}p+1)}\rt)^{\frac{\theta}n},
\]
with $q = \al(p-1)+1$, $\de = np - (n-p)q$, and an extremal functions is given the function $u(x) = (1 + |x|^{\frac{p}{p-1}})^{-\frac1{\al-1}}$, 

\item (ii) for $\al < 1$,
\begin{equation}\label{eq:sharpGN1}
\|u\|_{L^{\al (p-1)+1}(\R^n)} \leq GN(n,p,\al) \|\na u\|_{L^p(\R^n)}^{\theta} \|u\|_{L^{\al p}(\R^n)}^{1-\theta},\quad u\in C^\infty_0(\R^n),
\end{equation}
with $\theta =\frac{n(1-\al)}{(\al p + 1 -\al)(n-\al(n-p))}$, the sharp constant $GN(n,p,\al)$ is given by
\[
GN(n,p,\al) = \lt(\frac{p-q}{p\sqrt{\pi}}\rt)^{\theta}\lt(\frac{pq}{n(p-q)}\rt)^{\frac{\theta}p} \lt(\frac{pq}{\de}\rt)^{\frac{1-\theta}{\al p}} \lt(\frac{\Gamma(\frac{p-1}p \frac{\de}{p-q}+1) \Gamma(\frac n2+1)}{\Gamma(q\frac{p-1}{p-q}+1) \Gamma(n \frac{p-1}p+1)}\rt)^{\frac{\theta}n}, 
\]
with $q = \al(p-1)+ 1$, $\de =np - q(n-p) >0$, and an extremal functions is given the function $u(x) = (1 - |x|^{\frac{p}{p-1}})_+^{\frac1{1-\al}}$, where $a_+ = \max\{a,0\}$ denotes the positive part of a number $a$. 
\end{description}
We refer the reader to the paper of Cordero-Erausquin, Nazaret and Villani \cite{CNV} for a completely different proof of the sharp Sobolev inequality \eqref{eq:sharpSobolev} and the sharp Gagliardo--Nirenberg inequality \eqref{eq:sharpGN} and \eqref{eq:sharpGN1} by using the mass transportation method.

Finally, we recall the sharp Morrey--Sobolev inequality in $\R^n$. Given $p >n$, then for any function $u\in C_0^\infty(\R^n)$, the following inequality holds
\begin{equation}\label{eq:sharpMS}
\|u\|_{L^\infty(\R^n)} \leq b_{n,p} \Vol(\text{\rm supp}\,u)^{\frac{p-n}{np}} \|\na u\|_{L^p(\R^n)},
\end{equation}
here $\Vol$ denotes the Lebesgue measure of any measurable subset of $\R^n$, the sharp constant $b_{n,p}$ is given by
\[
b_{n,p} = n^{-\frac1p} \si_n^{-\frac1n}\lt(\frac{p-1}{p-n}\rt)^{\frac{p-1}p},
\]
and an extremal function is given by $u(x) = (1 -|x|^{\frac{p-n}{p-1}})_+$. For more about this inequality, the reader may consult \cite{Talentib}.

Let us go to prove Theorem \ref{Mainresult} and Theorem \ref{Maintheorem}.

\begin{proof}[Proof of Theorem \ref{Mainresult}]
Suppose $u$ is a function in $W^{1,p}(\H^n)$. Let us define two new functions $u_g^\sharp$ and $u_e^\sharp$ by \eqref{eq:usharpg} and \eqref{eq:usharpe} respectively. Theorem \ref{sosanh} implies
\[
\|\na_g u_g^\sharp\|_{L^p(\H^n)}^p - \lt(\frac{n-1}p\rt)^p \|u_g^\sharp\|_{L^p(\H^n)}^p \geq \|\na u_e^\sharp\|_{L^p(\R^n)}^p,
\]
for any $p\geq 2$ if $n=2$, and for any $p\geq \frac{2n}{n-1}$ if $n\geq 3$. Note that $\|u_g^\sharp\|_{L^p(\H^n)} = \|u\|_{L^p(\H^n)}$. Hence, applying P\'olya--Szeg\"o principle \eqref{eq:PSprinciple} and equality \eqref{eq:equalkey}, we get
\begin{equation}\label{eq:sosanh1}
\|\na_g u\|_{L^p(\H^n)}^p - \lt(\frac{n-1}p\rt)^p \|u\|_{L^p(\H^n)}^p \geq \|\na u_e^\sharp\|_{L^p(\R^n)}^p.
\end{equation}
Suppose that $n\geq 4$ and $\frac{2n}{n-1} \leq p < n$. Using the sharp Sobolev inequality \eqref{eq:sharpSobolev} for $u_e^\sharp$ and using the equality $\|u_e^\sharp\|_{L^{p^*}(\R^n)} = \|u\|_{L^{p^*}(\H^n)}$, we obtain the desired inequality \eqref{eq:LpHSM}.

Suppose $u\in W^{1,p}(\H^n)$ such that the equality in \eqref{eq:LpHSM} holds for $u$. Let $v= u^*$ the decreasing rearrangement function of $u$ on $[0,\infty)$, and define $u_g^\sharp$ and $u_e^\sharp$ by \eqref{eq:usharpg} and \eqref{eq:usharpe} respectively. Since the equality in \eqref{eq:LpHSM} holds for $u$, we must have $\|\na_g u\|_{L^p(\H^n)} = \|\na_g u_g^\sharp\|_{L^p(\H^n)}$ and
\[ 
\|\na_g u_g^\sharp\|_{L^p(\H^n)}^p - \lt(\frac{n-1}p\rt)^p \|u_g^\sharp\|_{L^p(\H^n)}^p = \|\na u_e^\sharp\|_{L^p(\R^n)}^p.
\]
From the proof of Theorem \ref{sosanh}, we see that the second condition implies
\[
\int_0^\infty |((v(s) s^{\frac1p})'|^p s^{p-1} ds =0.
\]
Thus, we have $v(s) = c s^{-\frac1p}$ for some constant $c\in \R$. However, $\int_0^\infty v(s)^p ds < \infty$ which forces $c =0$. This finishes our proof of Theorem \ref{Mainresult}.
\end{proof}

\begin{proof}[Proof of Theorem \ref{Maintheorem}]
The proof of Theorem \ref{Maintheorem} is similar with the one of Theorem \ref{Mainresult}. Suppose that $n\geq 4$ and $\frac{2n}{n-1} \leq p < n$. By \eqref{eq:sosanh1}, we can apply the sharp Gagliardo--Nirenberg inequalities \eqref{eq:sharpGN} and \eqref{eq:sharpGN1} for function $u_e^\sharp$ to derive the desired inequalities \eqref{eq:GN1} and \eqref{eq:GN2} as done for the inequality \eqref{eq:LpHSM}, respectively.

Suppose that $n\geq 2$ and $p >n$. We note that \eqref{eq:sosanh1} still holds under this condition. We now can apply the sharp Morrey--Sobolev inequality \eqref{eq:sharpMS} for $u_e^\sharp$ to yield the inequality \eqref{eq:HMS} with remark that $\|u_e^\sharp\|_{L^\infty(\R^n)} = \|u\|_{L^\infty(\H^n)}$ and $\Vol(\text{\rm supp}\, u_e^\sharp) = V(\text{\rm supp}\, u)$.
\end{proof}

We conclude this section by a remark in the case $\al \to 1^+$ of the inequality \eqref{eq:GN1}. Taking the limit as done in \cite{DDb}, we obtain the following Poincar\'e--Sobolev logarithmic inequality in $\H^n$ which is an extension of the optimal Euclidean $L^p-$Sobolev logarithmic inequality \cite{DDa,DDb} to the hyperbolic spaces. Suppose $u \in W^{1,p}(\H^n)$ with $\|u\|_{L^p(\H^n)} =1$, it holds
\begin{equation}\label{eq:logPS}
\int_{\B^n} |u|^p \ln (|u|^p) dV \leq \frac{n}{p} \ln\lt(\mathcal L_{n,p} \int_{\B^n} \lt(|\na_g u|_g^p - \lt(\frac{n-1}n\rt)^p |u|^p\rt)dV\rt)
\end{equation}
for any $n \geq 4$ and $\frac{2n}{n-1} \leq p < n$ with the constant $\mathcal L_{n,p}$ is given by
\[
\mathcal L_{n,p} = \frac{p}{n} \lt(\frac{p-1}e\rt)^{p-1} \pi^{\frac p2}\lt[\frac{\Gamma(\frac n2 +1)}{\Gamma(n \frac{p-1}p +1)}\rt]^{\frac pn}.
\]


\section{Other Sobolev inequalities in the hyperbolic spaces}
In this section, we establish several Sobolev inequalities in the hyperbolic spaces $\H^n$. These inequalities generalize the results of Mugelli and Talenti \cite{MT} in $\H^2$ to higher dimensional spaces. The main results of this section read as follows.

\begin{theorem}\label{MTHn}
Let $n \geq 2$ and $p\in [1,\infty)$. Then for any function $u\in W^{1,p}(\H^n)$, the following inequalities holds.

\item (i) If $p=1$ then
\begin{equation}\label{eq:MTp}
\lt(n-1\rt)^n \lt(\int_{\B^n} |u| dV\rt)^{n} + S(n,1)^n \lt(\int_{\B^n} |u|^{\frac{n}{n-1}} dV\rt)^{n-1} \leq \lt(\int_{\B^n} |\na_g u|_g dV\rt)^{n}.
\end{equation}

\item (ii) If $1 < p < n$ then
\begin{equation}\label{eq:MTp1}
\lt(\frac{n-1}p\rt)^n \lt(\int_{\B^n} |u|^p dV\rt)^{\frac np} + S(n,p)^n \lt(\int_{\B^n} |u|^{\frac{np}{n-p}} dV\rt)^{\frac{n-p}p} \leq \lt(\int_{\B^n} |\na_g u|_g^p dV\rt)^{\frac np}.
\end{equation}

\item (iii) If $n < p < \infty$ then
\begin{equation}\label{eq:MTp2}
\sup_{x\in \B^n} |u(x)| \leq C(n,p) \lt(\int_{\B^n} |\na_g u|_g^p dV\rt)^{\frac1p},
\end{equation}
with 
\[
C(n,p) =(2^{n-1} n\si_n)^{-\frac1p}\lt(\frac{\Gamma(\frac{p-n}{2(p-1)})\Gamma(\frac{n-1}{p-1})}{\Gamma(\frac{p+n-2}{2(p-1)})}\rt)^{\frac{p-1}p}.
\]
Furthermore, the equality holds in \eqref{eq:MTp2} if $u(x) = v(V(B_g(0,\rho(x))))$ with 
\begin{equation}\label{eq:vf}
v(r) = c\int_r^\infty \lt(\sinh \Phi^{-1}\lt(\frac s{\si_n}\rt)\rt)^{-\frac {p(n-1)}{p-1}}ds.
\end{equation}
\end{theorem}
Obviously, $(a+b)^{\al} \leq a^\al + b^\al$ for any $\al \in [0,1]$ and $a, b\geq 0$. As a consequence, the inequality \eqref{eq:MTp1} is weaker than the inequality \eqref{eq:LpHSM}. However,the inequality \eqref{eq:MTp1} is valid for any $n\geq 2$ and $1 < p < n$.

\begin{proof}
The part (i) follows from part (ii) by letting $p \downarrow 1$. We next prove part (ii). Let $u \in W^{1,p}(\H^n)$, we define two new functions $u_g^\sharp$ and $u_e^\sharp$ by \eqref{eq:usharpg} and \eqref{eq:usharpe} respectively. Denote $v =u^*$ and recall the function $\Phi$ from \eqref{eq:Phi}. By \eqref{eq:gradientHn}, we have
\begin{align*}
\int_{\B^n} |\na_g u^\sharp_g|_g^p dV &= (n\si_n)^p \int_0^\infty |v'(s)|^p \lt(\sinh \Phi^{-1}\lt(\frac{s}{\si_n}\rt)\rt)^{p(n-1)} ds\\
&=(n\si_n)^p \int_0^\infty |v'(s)|^p \lt(\lt(\sinh \Phi^{-1}\lt(\frac{s}{\si_n}\rt)\rt)^{n(n-1)}\rt)^{\frac pn} ds\\
&=(n\si_n)^p \int_0^\infty |v'(s)|^p \lt( k_{n,n}\lt(\frac s{\si_n}\rt) + \lt(\frac s{\si_n}\rt)^{n-1}\rt)^{\frac pn} ds\\
&\geq (n\si_n)^p \int_0^\infty |v'(s)|^p \lt( \lt(\frac{n-1}n\rt)^n\lt(\frac s{\si_n}\rt)^n + \lt(\frac s{\si_n}\rt)^{n-1}\rt)^{\frac pn} ds,
\end{align*}
here we use Lemma \ref{lowerboundk} to bound $k_{n,n}$ from below. It is easy to see that for $\al \in (0,1)$, it holds
\begin{equation}\label{eq:a1}
(a+b)^\al = \sup_{t\in [0,1]} (t^{1-\al} a^\al + (1-t)^{1-\al} b^\al).
\end{equation}
Consequently, for any $t\in [0,1]$, we get
\begin{align*}
\int_{\B^n} |\na_g u^\sharp_g|_g^p dV &\geq t^{1-\frac pn}(n-1)^p \int_0^\infty |v'(s)|^p s^p ds + (1-t)^{1-\frac pn} (n\si_n)^p \int_0^\infty |v'(s)|^p \lt(\frac s{\si_n}\rt)^{\frac{p(n-1)}n} ds\\
&\geq t^{1-\frac pn} \frac{(n-1)^p}{p^p} \int_0^\infty |v(s)|^p ds + (1-t)^{1-\frac pn}\|\na u_e^\sharp\|_{L^p(\R^n)}^p\\
&\geq t^{1-\frac pn} \frac{(n-1)^p}{p^p} \|u_g^\sharp\|_{L^p(\H^n)}^p + (1-t)^{1-\frac pn} S(n,p)^p \|u_e^\sharp\|_{L^{\frac{np}{n-p}}(\R^n)}^p,
\end{align*}
here we use the Hardy inequality in $[0,\infty)$ for the first inequality, and the sharp Sobolev inequality for the second inequality. Taking the supremum over $t \in [0,1]$ and using again \eqref{eq:a1} and the fact $\|u\|_{L^p(\H^n)} = \|u^\sharp_g\|_{L^p(\H^n)}$ and $\|u\|_{L^{p^*}(\H^n)} = \|u_e^\sharp\|_{L^{p^*}(\R^n)}$ with $p^* = \frac{np}{n-p}$, we obtain
\[
\int_{\B^n} |\na_g u^\sharp_g|_g^p dV \geq \lt[\lt(\frac{n-1}p\rt)^n \|u\|_{L^p(\H^n)}^n + S(n,p)^n \|u\|_{L^{\frac{np}{n-p}}(\H^n)}^n\rt]^{\frac pn},
\]
which implies \eqref{eq:MTp1} by the P\'olya--Szeg\"o principle in $\H^n$.

We next prove part (iii). Denote
\[
l(s) = \lt[\sinh \Phi^{-1} \lt(\frac s{\si_n}\rt)\rt]^{n-1}, \quad s \geq 0,
\]
where $\Phi$ is defined by \eqref{eq:Phi}. It is easy to check that $l(s) \sim s^{\frac{n-1}n}$ as $s \to 0$ and $l(s) \sim s$ as $s\to \infty$. Consequently, we get 
\[
\int_0^\infty l(s)^{-\frac p{p-1}} ds < \infty.
\]
Moreover, making the change $s = \si_n \Phi(t)$ with $ds =n\si_n (\sinh t)^{n-1}$, we have
\begin{align*}
\int_0^\infty l(s)^{-\frac p{p-1}} ds&= \int_0^\infty \lt[\sinh \Phi^{-1} \lt(\frac s{\si_n}\rt)\rt]^{-\frac{p(n-1)}{p-1}} ds = n\si_n \int_0^\infty (\sinh t)^{-\frac{n-1}{p-1}} dt\\
&= n\si_n 2^{\frac{n-p}{p-1}} \frac{\Gamma(\frac{p-n}{p-1})\Gamma(\frac{n-1}{2(p-1)})}{\Gamma(\frac{2p-n-1}{2(p-1)}}.
\end{align*}
Using the identity $\Gamma(x) \Gamma(x + \frac12) = 2^{1-2x} \Gamma(2x)$, we get
\begin{equation}\label{eq:cc}
\int_0^\infty l(s)^{-\frac p{p-1}} ds= n\si_n 2^{-\frac{n-1}{p-1}} \frac{\Gamma(\frac{p-n}{2(p-1)})\Gamma(\frac{n-1}{p-1})}{\Gamma(\frac{p+n-2}{2(p-1)})}.
\end{equation}
Denote $v =u^*$, we have $\lim_{s\to\infty} v(s) =0$, and hence
\[
v(r) = \int_r^\infty v'(s) ds = \int_{r}^\infty v'(s) l(s) l(s)^{-1} ds.
\]
Thank to H\"older inequality, we get
\begin{align*}
\sup_{x\in \H^n} |u(x)| = v(0) &\leq \lt(\int_0^\infty |v'(s)|^p l(s)^{p} ds\rt)^{\frac1p} \lt(\int_0^\infty l(s)^{-\frac p{p-1}} ds\rt)^{\frac{p-1}p}\\
&= \frac1{n\si_n} \lt(\int_0^\infty l(s)^{-\frac p{p-1}} ds\rt)^{\frac{p-1}p}\lt((n\si_n)^p\int_0^\infty |v'(s)|^p l(s)^{p} ds\rt)^{\frac1p}\\
&=C(n,p) \, \|\na u_g^\sharp\|_{L^p(\H^n)},
\end{align*}
here the second equality comes from \eqref{eq:gradientHn} and \eqref{eq:cc}. This proves \eqref{eq:MTp2}.

To check the sharpness of $C(n,p)$, we see that if $u(x) = v(V(B_g(0,\rho(x))))$ with $v$ defined by \eqref{eq:vf}, then $u^* = v$. Hence $\sup_{x\in \H^n} |u(x)| =v(0)$. Moreover, for such a choice of function $v$, we have equality in the H\"older inequality above. This proves the sharpness of $C(n,p)$ and the equality holds for this function $u$. 
\end{proof}

\section*{Acknowledgements}
The author would like to thank two anonymous reviewers for their useful and constructive comments and suggestions which improve the presentation of this paper.


\begin{thebibliography}{9999}
\bibitem{Aubin}
T. Aubin, \emph{Probl\`emes isop\'erim\'etriques et espaces de Sobolev\text}, J. Differential Geom., {\bf 11} (1976) 573--598.

\bibitem{ADH}
T. Aubin, O. Druet, and E. Hebey, \emph{Best constants in Sobolev inequalities for compact manifolds of nonpositive curvature\text}, C. R. Acad. Sci. Paris S\'er. I Math., {\bf 326} (1998) 1117--1121.

\bibitem{AL}
T. Aubin, and Y. Y. Li, \emph{On the best Sobolev inequality\text}, J. Math. Pures Appl., {\bf 78} (1999) 353--387.

\bibitem{Ba}
A. Baernstein II, \emph{A Unifed Approach to Symmetrisation, Partial Differential Equations of Elliptic Type\text}, Cortona, 1992, in: Sympos. Math., vol. XXXV, Cambridge Univ. Press, Cambridge, 1994, pp. 47--91.

\bibitem{BFL}
R. D. Benguria, R. L. Frank, and M. Loss, \emph{The sharp constant in the Hardy--Sobolev--Maz'ya inequality in the three dimensional upper half-space\text}, Math. Res. Lett., {\bf 15} (2008) 613--622.

\bibitem{BDGG}
E. Berchio, L. D'Ambrosio, D. Ganguly, and G. Grillo, \emph{Improved $L^p-$Poincar\'e inequalities on the hyperbolic space\text}, Nonlinear Anal., {\bf 157} (2017) 146--166.

\bibitem{BG}
E. Berchio, and D. Ganguly, \emph{Improved higher order Poincar\'e inequalities on the hyperbolic space via Hardy--type remainder terms\text}, Commun. Pure Appl. Anal., {\bf 15} (2016) 1871--1892.

\bibitem{BGG}
E. Berchio, D. Ganguly, and G. Grillo, \emph{Sharp Poincar\'e--Hardy and Poincar\'e--Rellich inequalities on the hyperbolic space\text}, J. Funct. Anal., {\bf 272} (2017) 1661--1703. 

\bibitem{CMa}
J. Ceccon, and M. Montenegro, \emph{Optimal $L^p-$Riemannian Gagliardo--Nirenberg inequalities\text}, Math. Z., {\bf 258} (2008) 851--873.

\bibitem{CMb}
J. Ceccon, and M. Montenegro, \emph{Optimal Riemannian $L^p-$Gagliardo--Nirenberg inequalities revisited\text}, J. Differential Equations, {\bf 254} (2013) 2532--2555.

\bibitem{CMc}
J. Ceccon, and M. Montenegro, \emph{Sharp $L^p-$entropy inequalities on manifolds\text}, J. Funct. Anal., {\bf 269} (2015) 1591--1619.

\bibitem{CMd}
J. Ceccon, and M. Montenegro, \emph{Sharp constants in Riemannian $L^p-$Gagliardo--Nirenberg inequalities\text}, J. Math. Anal. Appl., {\bf 433} (2016) 260--281.

\bibitem{CNV}
D. Cordero-Erausquin, B. Nazaret, and C. Villani, \emph{A mass--transportation approach to sharp Sobolev and Gagliardo--Nirenberg inequalities\text}, Adv. Math., {\bf 182} (2004) 307--332.

\bibitem{DDa}
M. del Pino, and J. Dolbeault, \emph{Best constants for Gagliardo-Nirenberg inequalities and applications to nonlinear diffusions\text}, J. Math. Pures Appl., {\bf 81} (2002) 847--875. 

\bibitem{DDb}
M. del Pino, and J. Dolbeault, \emph{The optimal Euclidean $L_p-$Sobolev logarithmic inequality\text}, J. Funct. Anal., {\bf 197} (2003) 151--161.

\bibitem{Druet}
O. Druet, \emph{The best constants problem in Sobolev inequalities\text}, Math. Ann., {\bf 314} (1999) 327--346.

\bibitem{DH}
O. Druet, and E. Hebey, \emph{AB program in geometric analysis: sharp Sobolev inequalities and related problems\text}, Mem. Amer. Math. Soc., {\bf 160} (2002), no. 761, viii+98 pp. 

\bibitem{FMT}
S. Filippas, V. G. Maz'ya, and A. Tertikas, \emph{Sharp Hardy--Sobolev inequalities\text}, C. R. Acad. Sci. Paris Ser. I, {\bf 339} (2004) 483--486.

\bibitem{FMT1}
S. Filippas, V. G. Maz'ya, and A. Tertikas, \emph{Critical Hardy--Sobolev inequalities\text}, J. Math. Pures Appl., {\bf 87} (2007) 37--56.

\bibitem{FL}
R. L. Frank, and M. Loss, \emph{Hardy--Sobolev--Maz'ya inequalities for arbitrary domains\text}, J. Math. Pures Appl., {\bf 97} (2011) 39--54.

\bibitem{Hebey}
E. Hebey, \emph{Nonlinear analysis on manifolds: Sobolev spaces and inequalities\text}, Courant Lecture Notes in Mathematics Vol 5, AMS, Providence, RI, 1999.

\bibitem{HVa}
E. Hebey, and M. Vaugon, \emph{The best constant problem in the Sobolev embedding theorem for complete Riemannian manifolds\text}, Duke Math. J., {\bf 79} (1995) 235--279.

\bibitem{HVb}
E. Hebey, and M. Vaugon, \emph{Meilleures constantes dans le th\'eor\`eme d'inclusion de Sobolev\text}, Ann. Inst. H. Poincar\'e Anal. Non lin\'eaire, {\bf 13} (1996) 57--93.

\bibitem{LuYang}
G. Lu, and Q. Yang, \emph{Sharp Hardy--Adams inequalities for bi--Laplacian on hyperbolic space of dimension four\text}, Adv. Math., {\bf 319} (2017) 567--598.

\bibitem{LY}
G. Lu, and Q. Yang, \emph{Hardy--Sobolev--Maz'ya inequalities for higher order derivatives on half spaces\text}, preprint, arXiv:1703.08171.

\bibitem{MS}
G. Mancini, and K. Sandeep, \emph{On a semilinear elliptic equation in $\H^n$\text}, Ann. Scoula Norm. Sup. Pisa Cl. Sci. (5), {\bf 7} (2008) 635--671.

\bibitem{MS2010}
G. Mancini, and K. Sandeep, \emph{Moser--Trudinger inequality on conformal discs\text}, Commun. Contemp. Math., {\bf 12} (2010) 1055--1068.

\bibitem{MST2013}
G. Mancini, K. Sandeep, and C. Tintarev, \emph{Trudinger--Moser inequality in the hyperbolic space $\mathbb H^n$\text}, Adv. Nonlinear Anal., {\bf 2} (2013) 309--324.

\bibitem{Maz'ya}
V. G. Maz'ya, \emph{Sobolev spaces\text}, Springer Verlag, Berlin, New York, 1985.

\bibitem{MT}
F. Mugelli, and G. Talenti, \emph{Sobolev inequalities in $2-$dimensional hyperbolic space\text}, General inequalities, 7 (Oberwolfach, 1995), 201--216, Internat. Ser. Numer. Math., 123, Birkh\"auser, Basel, 1997.


\bibitem{Nguyen2017}
V. H. Nguyen, \emph{The sharp Hardy--Moser--Trudinger inequality in dimension $n$\text}, preprint.

\bibitem{Nguyen2017a}
V. H. Nguyen, \emph{Improved Moser--Trudinger type inequalities in the hyperbolic space $\H^n$\text}, Nonlinear Anal., {\bf 168} (2018) 67--80.

\bibitem{Talentia}
G. Talenti, \emph{Best constants in Sobolev inequality\text}, Ann. Mat. Pura Appl., {\bf 110} (1976) 353--372.

\bibitem{Talentib}
G. Talenti, \emph{On isoperimetric theorems of mathematical physics\text}, in: P. M. Gruber, J. M. Wills (Eds.), Handbook of convex geometry, Vol. A, B, 1131--1147, North-Holland, Amsterdam, 1993.

\bibitem{TT}
A. Tertikas, and K. Tintarev, \emph{On existence of minimizers for the Hardy--Sobolev--Maz'ya inequality\text}, Ann. Mat. Pura Appl. (4), {\bf 186} (2007) 645--662.

\bibitem{WY2012}
G. Wang, and D. Ye, \emph{A Hardy--Moser--Trudinger inequality\text}, Adv. Math., {\bf 230} (212) 294--320.


\end{thebibliography}
\end{document}